\documentclass[12pt]{amsart}
\usepackage{Baskervaldx}
\usepackage[baskervaldx,cmintegrals,bigdelims,vvarbb]{newtxmath}
\usepackage[cal=cm]{mathalfa}
\usepackage{amssymb, graphicx, caption, subcaption, amsmath, array, pdfsync, tikz, tikz-cd, color, url, float, enumerate, geometry}
\usepackage[all]{xy}

\usetikzlibrary{calc}

\title{The degree of irrationality of hypersurfaces in various Fano varieties\vspace{-2ex}}
\author{David Stapleton and Brooke Ullery\vspace{-2ex}}

\usepackage[pagebackref]{hyperref}
\hypersetup{
  colorlinks= true,
  linkcolor=[RGB]{140,49,0},
  urlcolor=[RGB]{140,49,0},
  citecolor=[RGB]{140,49,0}
}

\newtheorem{theorem}{Theorem}
\numberwithin{theorem}{section}

\newtheorem{lemma}[theorem]{Lemma}
\newtheorem{proposition}[theorem]{Proposition}

\newcommand{\hr}[2]{\hyperref[#1]{#2}}

\newtheorem*{thmA}{Theorem A}
\newtheorem*{thmB}{Theorem B}
\newtheorem*{thmC}{Theorem C}
\newtheorem*{thmD}{Theorem D}
\newtheorem*{thmE}{Theorem E}

\theoremstyle{definition}
\newtheorem{remark}[theorem]{Remark}
\newtheorem{question}[theorem]{Question}
\newtheorem{problem}[theorem]{Problem}
\newtheorem{example}[theorem]{Example}
\newtheorem{definition}[theorem]{Definition}

\def\ZZ{{\mathbb Z}}

\def\CC{{\mathbb C}}

\def\Quad{{\mathbf{Q}}}
\def\Grass{{\mathbf{G}}}
\def\PP{{\mathbf{P}}}

\def\Oc{{\mathcal O}}

\def\Sc{{\mathcal S}}

\def\Grkm{{\mathrm{Gr}(k,m)}}
\def\Flkm{{\mathrm{Fl}(k-1,k,m-1)}}
\def\Flkkm{{\mathrm{Fl}(k-1,k+1,m)}}

\def\Grtq{{\Fano(\Quad)}}
\def\min{{\mathrm{min}}}
\def\max{{\mathrm{max}}}
\def\Fano{{\mathrm{Fano}}}
\def\Alb{{\mathrm{Alb}}}
\def\ord{{\mathrm{ord}}}

\def\dra{{\dashrightarrow}}
\def\ra{{\rightarrow}}

\def\cl{{\colon}}

\def\deg{{\mathrm{deg}}}

\def\irr{{\mathrm{irr}}}

\def\gon{{\mathrm{gon}}}
\def\min{{\mathrm{min}}}
\def\id{{\mathrm{id}}}
\def\bva{{\mathrm{(BVA)}_p}}
\def\BVA{{\mathrm{(BVA)}}}

\def\Jac{{\mathrm{Jac}}}
\def\bir{{\simeq_{\mathrm{bir}}}}

\def\dim{{\mathrm{dim}}}

\def\Bs{{\mathrm{Bs}}}
\def\inc{{\mathrm{Inc}}}
\def\Sym{{\mathrm{Sym}}}
\def\cb{{\mathrm{CB}}}

\def\Oc{{\mathcal O}}

\begin{document}
\maketitle

\section*{Introduction}
\thispagestyle{empty}

The \textit{degree of irrationality} of an $n$-dimensional algebraic variety $X$, denoted $\irr(X)$, is the minimal degree of a dominant rational map
$$
\phi\cl X \dra \PP^n.
$$
The aim of this paper is to compute the degree of irrationality of hypersurfaces in various Fano varieties: quadrics, cubic threefolds, cubic fourfolds, complete intersection threefolds of type (2,2), products of projective spaces, and Grassmannians. Throughout we work with varieties over $\CC$.

Recently there has been a great deal of interest in understanding different measures of irrationality of higher dimensional varieties. Bastianelli, Cortini, and De Poi conjectured  (\cite[Conj. 1.5]{BastCortDe}) that if $X$ is a very general $d$ hypersurface
$$
X=X_d \subset \PP^{n+1}
$$
with $d\ge 2n+1$, then $\irr(X) =d-1$. They proved their conjecture in the case $X$ is a surface or threefold. This conjecture was proved in full by Bastianelli, De Poi, Ein, Lazarsfeld, and the second author in (\cite{ELU}). Gounelas and Kouvidakis (\cite{GounKouv}) computed the covering gonality and the degree of irrationality of the Fano surface of a generic cubic threefold. Bastianelli, Ciliberto, Flamini, and Supino (\cite{BCFS}) computed the covering gonality of a  very general hypersurface in $\PP^{n+1}$. Recently, Voisin (\cite{VoisinCovgon}) proved that the covering gonality of a very general $n$-dimensional abelian variety goes to infinity with $n$.

In this paper we show that the ideas in the proof of ~\cite[Thm. C]{ELU} can be extended to compute the degree of irrationality of hypersurfaces in many Fano varieties. For example, let $\Quad\subset \PP^{n+2}$ be a smooth quadric in projective space.

\begin{thmA}
Let
$$
X = X_d \subset \Quad \subset \PP^{n+2}
$$
be a very general hypersurface in $\Quad$ with $X \in |\Oc_\Quad(d)|$. If $d\ge 2n$ then $\irr(X)=d.$
\end{thmA}

We have other results for hypersurfaces in cubic threefolds and cubic fourfolds.

\begin{thmB}
Let
$$
X = X_d \subset Z \subset \PP^{n+2}
$$
be a smooth complete intersection of type $(3,d)$ in a smooth cubic hypersurface.

\begin{enumerate}
\item If $n= 2$ and $d\ge 8$ then
$$
\irr(X)= 
\left\{\begin{array}{l}
2d-2 \text{ if $X$ contains a line,}\\
2d \text{ otherwise,} \end{array}\right.
$$
 and any rational map $X \dashrightarrow \PP^2$ with degree equal to $\irr(X)$ is birationally equivalent to projection from a line in $Z$.

\item If $n =3$, $d\ge 13$ and $X$ is very general in $|\Oc_Z(d)|$, then $\irr(X) = 2d$.
\end{enumerate}
\end{thmB}

Now let $Z = Z_{(2,2)} \subset \PP^5$ be a smooth complete intersection of two quadrics.
\begin{thmC}
Let
$$
X=X_d\subset Z
$$
be a smooth surface in $Z$ with $X\in |\Oc_Z(d)|$. If $d\ge 8$ then
$$\begin{array}{c}
\irr(X)= 
\left\{\begin{array}{l}
2d-2 \text{ if $X$ contains a plane conic,}\\
2d-1 \text{ if $X$ contains a line and no conics,}\\
2d \text{ otherwise.}\end{array}\right.
\end{array}$$
Moreover, any rational map $X\dra \PP^n$ with degree equal to $\irr(X)$ is birationally equivalent to the projection from a plane contained in a quadric in the linear series $|I_Z(2)|$.
\end{thmC}

Furthermore, we compute the degree of irrationality of hypersurfaces in Grassmannians. Let
$$
\Grass = \Grkm\subset \PP
$$
be the Grassmannian of $k$ planes in an $m$ dimensional vector space embedded via its Pl\"ucker embedding. Assume $k\ne 1, n-1$ (the excluded cases are covered by ~\cite[Thm. C]{ELU}).

\begin{thmD}
Let
$$
X = X_d \subset \Grass
$$
be a very general hypersurface with $X\in |\Oc_\Grass(d)|$. If $d\ge 3m-5$ then $\irr(X) = d$.
\end{thmD}

Finally, let $\PP = \PP^{n_1}\times \cdots \times \PP^{n_k}$ be a product of $k$ projective spaces with $k\ge 2$.

\begin{thmE}
Let
$$
X = X_{d_1,\dots,d_k}\subset \PP
$$
be a very general hypersurface with $X\in |\Oc_\PP(d_1,\dots,d_k)|$. Let $p$ be the minimum of $\{d_i-m_i-1\}$. If $p\ge \max\{m_i\}$ then $\irr(X) = \min\{d_i\}$.
\end{thmE}

A recurring theme throughout the paper is that the positivity of the canonical linear series helps to control the degree of irrationality. For example, given a dominant rational map:
$$
\phi\cl X \dra \PP^n,
$$
every finite fiber of $\phi$ satisfies the \textit{Cayley-Bacharach condition} (Definition 1.6) with respect to the canonical linear series $|\omega_X|$. This affects the possible projective configurations of the fibers. As a consequence, if $Z \subset \PP$ is one of the Fano varieties above in its natural projective embedding, and $X\subset Z$ is a hypersurface of sufficiently high degree, then we will see that any fiber of $\phi$ must lie on a low degree curve $C\subset Z$ (in the cases we consider, $C$ will always have degree $\le 2$).

This allows us to study low degree maps to $\PP^n$ by studying the geometry of low degree curves on these Fano varieties. In some cases (when $Z$ is a cubic threefold, or a (2,2) complete intersection threefold), the geometry of the spaces parametrizing low degree curves is explicit enough to complete the computation of the degree of irrationality of $X\subset Z$. In the other cases, we use the assumption that $X$ is very general and  follow the ideas of \cite[Prop. 3.8]{ELU} to conclude the proofs.
 
We would like to thank Enrico Arbarello, Lawrence Ein, Sam Grushevsky, Joe Harris, James Hotchkiss, Rob Lazarsfeld, Giulia Sacc\`a, John Sheridan, Ian Shipman, and Ruijie Yang for interesting and helpful conversations. The research of the second author was partially supported by an NSF Postdoctoral Fellowship, DMS-1502687.


\section{Background}

In this section we introduce the main definitions, and recall some known results. There is a nice introduction to these ideas in \cite{ELU}, and we refer the interested reader there for more details. At the end of this section, we also prove a preliminary result about points in projective space satisfying the Cayley-Bacharach condition.

\begin{definition}
Let $X$ be an $n$-dimensional algebraic variety. The \textit{degree of irrationality of $X$}, denoted $\irr(X)$, is the minimal degree of a rational map
$$
\phi\colon X \dra \PP^n.
$$
\end{definition}

The degree of irrationality of $X$ is a birational invariant of $X$. It is possible to give lower bounds on $\irr(X)$ by understanding the birational positivity of $K_X$, in an appropriate sense.

\begin{definition}
Let $L$ be a line bundle on a variety $X$. 
\begin{enumerate}
\item We say $L$ is \textit{$p$-very ample} if for all 0-dimensional subschemes $\xi \subset X$ of length $p+1$, the restriction map
$$
H^0(X,L) \ra H^0(X,L|_\xi)
$$
is surjective.
\item We say $L$ \textit{satisfies property} $\bva$ if there exists a nonempty open set $\emptyset\ne U\subset X$ such that for all 0-dimensional subschemes $\xi \subset U$ of length $p+1$, the restriction map
$$
H^0(X,L) \ra H^0(X,L|_\xi)
$$
is surjective.
\end{enumerate}
\end{definition}

\begin{example}
A line bundle $L$ satisfies $\BVA_0$ if and only if $L$ is effective. Moreover, $L$ satisfies $\BVA_1$ if and only if the linear series $|L|$ maps $X$ birationally onto its image in projective space.
\end{example}

\begin{example}
If $L=\Oc_\PP(p)$, then $L$ is \textit{$p$-very ample} and in particular, $L$ satisfies $\bva$.
\end{example}

\begin{theorem}[{\cite[Thm. 1.10]{ELU}}]
Let $X$ be a smooth projective variety and suppose that $\omega_X$ satisfies property $\bva$, then
$$
\irr(X)\ge p+2.
$$
\end{theorem}

One fundamental fact that we will use is that the fibers of a dominant rational map
$$
\phi\colon X \dra \PP^n
$$
lie in special position, in the sense that they satisfy the \textit{Cayley-Bacharach condition}.

\begin{definition}
Let $\Sc\subset \PP$ be a finite subset of projective space. We say that the set $\Sc$ \textit{satisfies the Cayley-Bacharach condition with respect to} $|\Oc_\PP(m)|$ (or just $\Sc$ \textit{satisfies} $\cb(m)$) if any divisor $D\in |\Oc_\PP(m)|$ which contains all but one point of $\Sc$, contains all of $\Sc$.
\end{definition}

Let $X\subset \PP$ be a smooth $n$-dimensional subvariety of projective space. Assume that $\omega_X = \Oc_X(m)$ for some $m$. The following proposition was proven by Bastianelli, Cortini, and De Poi.

\begin{proposition}[{\cite[Prop. 2.3]{BastCortDe}}]\label{BCD} 
\begin{enumerate}
\item Assume that $\Gamma\subset X\times \PP^n$ is a reduced subscheme of pure dimension $n$. Assume that $y\in \PP^n$ is a smooth point for the projection
$$ \pi_2|_\Gamma \colon \Gamma \ra \PP^n.$$
Then the set $\Sc = \pi_1((\pi_2|_\Gamma)^{-1}(y))$ satisfies $\cb(m).$
\item In the special case when $\Gamma$ is the graph of a rational map $\phi\colon X \dra \PP^n$, (1) implies that a general fiber of $\phi$ satisfies $\cb(m)$.
\end{enumerate}
\end{proposition}

Furthermore, those authors show that there are strong geometric consequences imposed on small sets $\Sc\subset \PP$ which satisfy $\cb(m)$.

\begin{lemma}[{\cite[Lem. 2.4]{BastCortDe}}]\label{BCDline}
Let $n\ge 2$ and let $\Sc\subset \PP$ be a set of $r$ points in projective space which satisfy $\cb(m)$. Then $r \ge m+2$. Moreover, if $r \le 2m+1$ then all the points in $\Sc$ lie on a line $\ell\subset \PP^n$.
\end{lemma}

\noindent Finally, in order to prove Theorem B and Theorem C we need a mild generalization of Lemma 1.8. We encourage the casual reader to skip the proof of the following theorem.

\begin{theorem}\label{deg2curve}
Let $\Sc$ be a set of $r$ points in projective space which satisfy $\cb(m)$. If
$$
r\le (5/2) m + 1
$$
then $\Sc$ is contained in a curve $C$ with $\deg(C)\le 2$ (either a line, a plane conic, or a union of two lines).
\end{theorem}

To prove Theorem 1.9, we start with the case when $\Sc\subset \PP^2$ is contained in a plane.

\begin{lemma}
Let $\Sc\subset \PP^2$ be a set of $r$ points which satisfy $\cb(m)$. If
$$
r\le (5/2)m+1
$$
then $\Sc$ is contained in a curve $C\subset \PP^2$ with degree $\le 2$.
\end{lemma}

\begin{proof}
We proceed by induction. First we need to take care of all cases when $m\le 3$.

When $m=1$, then $r\le 3$, so there is a conic containing all points in $\Sc$. When $m=2$, then $r\le 6$. There must be a conic $C$ through 5 of the points in $\Sc$, and because $\Sc$ satisfies $\cb(2)$ we know $\Sc\subset C$.

Let $m=3$ and first assume there is a line containing $\rho\ge 3$ points. The remaining $r-\rho$ points satisfy $\cb(2)$ and thus lie on a line by Lemma 1.8, so $\Sc$ is contained in the union of 2 lines. Now assume no 3 points lie on a line, and take a conic $C$ which contains $\rho$ points where $\rho\ge 5$. Thus there are $r-\rho\le 3$ remaining points. These points satisfy $\cb(1)$. Thus by Lemma 1.8, we can conclude that $\rho=8$ and thus all the points of $\Sc$ must lie on the conic $C$.

Proceeding by induction, let
$$
C_1\text{ be either } 
\left\{\begin{array}{l}
\text{1. a line in $\PP^2$ such that $\# C_1\cap \Sc =\rho \ge 3$, or }\\
\text{2. a conic in $\PP^2$ such that $\# C_1\cap \Sc =\rho \ge 5$.}
\end{array}\right.
$$
In case 1, the remaining $r-\rho$ points satisfy $\cb(m-1)$, and
$$
r-\rho \le (5/2)(m-1)+1.
$$
In case 2, the remaining $r-\rho$ points satisfy $\cb(m-2)$ and
$$
r-\rho \le (5/2)(m-2)+1.
$$
In either case, by induction, there is a curve $D\subset \PP^2$ containing $\Sc$ which is the union of lines and conics and satisfies $\deg(D)\le 4$. Moreover, if $C_1$ is a line then $\deg(D)\le 3.$

Now assume that $D$ contains a line $C_1\subset D$ and $C_1$ contains a point in $\Sc$ which is not contained in any other component of $D$. Then the points on $C_1$ which don't lie on another component of $D$ satisfy $\cb(m-3)$, thus as $m\ge 4$ by Lemma 1.8 there are at least 3 points on $C_1$. Thus by the previous paragraph we can assume that $\deg(D)\le 3$. So we are in the situation where all the points are on a line $C_1$ and a conic $C_2$ and $D=C_1\cup C_2$. Suppose there are a total of $\rho'$ points which do not lie on $C_1$. These points satisfy $\cb(m-1)$. If $\rho' \le 2(m-1)+1=2m-1$ then we are done by Lemma 1.8. So assume for contradiction that
$$
\rho' \ge 2m.
$$
Returning to the points on $C_1$ we see there are at most $r-\rho'$ points on $C_1$ not contained in $C_2$. These points satisfy $\cb(m-2)$, therefore by Lemma 1.8
$$
m\le r-\rho'.
$$
However, the previous inequality implies
$$
r-\rho' \le r-2m \le (5/2) m+1-2m \le  m/2 +1.
$$
The right hand side is less than $m$ (for $m\ge 4$), which gives the contradiction.

Now we can assume that $\Sc \subset D = C_1 \cup C_2$ is the union of two smooth conics. Assume without loss of generality that $C_1$ contains at least half (but not all) of the points in $\Sc$. If $\rho$ is the number of points in $\Sc$ which are not contained in $C_1$, then
$$
\rho\le r/2 \le (5/4)m+1/2.
$$
Moreover, these $\rho$ points satisfy $\cb(m-2)$, and when $m\ge 2$ we know that
$$
\rho\le (5/4)m+1/2\le 2(m-2)+1=2m-1.
$$
Therefore, by Lemma 1.8 we know there are at least 4 points on a line so we are done by the first case.
\end{proof}

\begin{proof}[Proof of Theorem 1.9]
By Lemma 1.10, we can assume that the points $\Sc$ are not contained in a plane. Again we plan to proceed by induction, and we need to start by checking the cases $m=1$ and 2 (which we leave to the reader).

Let $\rho$ be the maximum number of points in $\Sc$ which are contained in a plane. Then we can assume $\rho\ge 3$. The remaining $r-\rho$ points satisfy CB(m-1), and we have the inequality
$$
r-\rho \le (5/2)(m-1)+1.
$$
So by induction the remaining points lie on a plane conic or on a pair of skew lines.

In the case when the remaining points lie on a plane conic, by the definition of $\rho$ we know that $r-\rho \le \rho$, and thus
$$
r-\rho \le r/2 \le (5/4) m + 1/2.
$$
As $m\ge 3$ we have that
$$
(5/4) m + 1/2 \le 2(m-1) + 1
$$
and by Lemma 1.8 we have that all the $r-\rho$ points lie on a line $C_1$. By Lemma 1.8 there are at least $m+1$ points of $\Sc$ on $C_1$. Therefore, there are at most $r-m-1$ points not on $C_1$ and these points satisfy CB($m-1$). Combining our inequalities we have
$$
r-m-1 \le (5/2)m+1-m-1 \le (3/2)m \le 2(m-1)+1.
$$
Thus by Lemma 1.8 the points which are not contained on $C_1$ are contained in another line $C_2$, which proves the first case.

The last case to take care of is when $m\ge 3$ and the remaining $r-\rho$ points lie on two skew lines $C_1$ and $C_2$. Suppose there are $\rho'$ points on $C_1$ and $\rho''$ points on $C_2$ (and thus $r-\rho = \rho' + \rho''$). In this case the $\rho'$ points on $C_1$ and the $\rho''$ points on $C_2$ both satisfy CB(m-2). Thus by Lemma 1.8 we have, $\rho',\rho'' \ge m$. Moreover, by the definition of $\rho$ it is clear that $\rho\ge \max\{\rho',\rho''\}\ge m$ (points on a line lie on a plane). Therefore, we have
$$
3m\le \rho+\rho'+\rho'' = r \le (5/2)m+1
$$
which is a contradiction.
\end{proof}

\begin{remark}
In the setting of Theorem 1.9, if the set $\Sc$ lies on the union of two lines, it is easy to show that each line contains at least $m+1$ points.
\end{remark}


\section{The degree of irrationality of a hypersurface in a quadric}

Let  $\Quad\subset \PP^{n+2}$ be a smooth $(n+1)$-dimensional quadric in projective space. The aim of this section is to prove Theorem A, that is if $X\in |\Oc_\Quad(d)|$ is a very general hypersurface with $n\ge 1$ and $d\ge 2n$, then $\irr(X) = d$.

\begin{example}
When $n=1$, i.e. $X \subset \Quad \subset \PP^3$ is a curve in a smooth quadric in $\PP^3$, then we know $\Quad\cong \PP^1 \times \PP^1$. Projection onto either factor
$$
\phi_0\cl X \ra \PP^1
$$
gives a degree $d$ map to $\PP^1$. By adjunction $\omega_X = \Oc_X(d-2)$ is $(d-2)$-very ample. By Theorem 1.5
$$
d\ge \irr(X) = \gon(X) \ge d-2+2=d.
$$
In higher dimensions, this can be generalized as it is always possible to project from a line $\ell\subset \Quad$
\end{example}

\begin{proposition}
If $X\in |\Oc_\Quad(d)|$ is any hypersurface, then there exists
$$
\phi_0\cl X \dra \PP^n
$$
such that $\deg(\phi_0) = d$. Therefore $\irr(X) \le d$.
\end{proposition}

\begin{proof}
Choose a line $\ell \subset \Quad$ which meets $X$ properly. Let $p_\ell \cl \PP^{n+2}\dra \PP^n$ denote the linear projection from $\ell$ and set $\phi_0 = p_\ell|_X$. The closure of each fiber of $p_\ell$ is a plane $\PP^2 \subset \PP^{n+2}$ containing $\ell$. Thus, for a general such plane we can compute
$$
\begin{array}{rcl}
\deg(\phi_0)&=&\left(\text{length of }\PP^2\cap X\right)-\left(\text{length of }\PP^2\cap X \text{ supported on }\Bs(\phi_0)\right)\\
&=&\mathrm{length}(\PP^2\cap X) - \mathrm{length}(\ell\cap X)=2d-d=d
\end{array}
$$
\end{proof}

In the case of \cite[Thm. C]{ELU}, those authors prove that if $X\subset\PP^{n+1}$ is a very general hypersurface with $d\ge 2n+2$, then any degree $d-1$ map is given by projection from a point up to post composition with a birational automorphism of $\PP^n$. Such a uniform description is not possible for quadrics in all dimensions. Already in Example 2.1 we see that there are two possible projections $X\ra \PP^1$. So maps computing the gonality are not unique already when $n=1$.

\begin{example}[Another map realizing $\irr(X)$ when $n$ is odd.] Let $X\subset \Quad\subset \PP^{n+2}$ be as above and assume that $n=2k-1$ is odd. There exist non-intersecting linear subspaces of dimension $k$:
$$
\PP(V),\PP(W)\subset \Quad.
$$
These are \textit{maximal isotropic subspaces} which are in the same family. As they do not intersect we may write $\PP^{n+2} = \PP(V\oplus W).$

The rational map
\begin{center}
$p_{V,W} \cl \PP(V\oplus W) \dra \PP(V)\times \PP(W).$\\
$[v\oplus w]\mapsto [v]\times [w]$.
\end{center}
\noindent maps $\Quad$ onto a rational divisor $B\subset \PP(V\times W)$ of type $(1,1)$, and contracts lines in $\Quad$ of the form
$$
\ell = \{ [sv\oplus tw]| [s:t]\in \PP^1\}.
$$
The restriction $\phi_1 =p_{V,W}|_X$
has degree $d$ if $V$ and $W$ are chosen generally.
\end{example}

In order to reach a contradiction and prove Theorem A we assume that there exists a map
$$
\phi\cl X \dra \PP^n
$$
with $\delta=\deg(\phi)<d$. First, we note that all fibers of $\phi$ must lie on a line $\ell \subset \Quad$.

\begin{lemma}
If $d\ge 2n$ and $(d,n)\ne (2,1)$ then a general fiber of $\phi$ lies on a line $\ell\subset \PP^{n+2}$ which is contained in $\Quad.$
\end{lemma}

\begin{proof}
By adjunction $\omega_X = \Oc_X(d-n-1)$. The assumption that $d\ge 2n$ implies that $\delta \le 2(d-n-1) +1.$ Thus as a general fiber of $\phi$ satisfies Cayley-Bacharach with respect to $|\omega_X|$ by Lemma 1.8 a general fiber of $\phi$ must lie on a line $\ell\subset\PP^{n+2}$. And assuming $(d,n)\ne (2,1)$, by Theorem 1.5 at least 3 points lie on the line, so as a consequence of Bezout's theorem we have $\ell\subset \Quad$.
\end{proof}

By the previous lemma, a general point $y\in\PP^n$ parameterizes a line $\ell_y\subset \Quad$ (the span of the fiber $\phi^{-1}(y)$). This induces a rational map $\PP^n \dra \Grtq$, where $\Grtq$ is the Fano variety of lines contained in $\Quad$ (the \textit{orthogonal Grassmannian}). Resolving the map gives
$$
f\cl B\ra \Grtq
$$
where $B$ is a smooth and rational projective variety. The map $f$ gives rise to the following fundamental diagram whose terms are defined below:

\begin{equation}
\begin{tikzcd}
X\arrow[d,hook]& X'\arrow[l]\arrow[d,hook]\\
\Quad & F\arrow[l,"\pi"]\arrow[d,"\psi"]\\
& B\arrow[r,"f"] & \Grtq.
\end{tikzcd}
\end{equation}

\noindent Here $\psi\cl F \ra B$ is the $\PP^1$-bundle defined as the pullback of the natural $\PP^1$-bundle over $\Grtq$. Thus $F$ comes with a natural projection $\pi\cl F \ra \Quad$. The fact that $X$ is not uniruled implies $\pi$ is generically finite. To define $X'$ consider the rational map:
$$
\id_X \times \varphi\cl X \dra X \times B\subset \Quad \times B
$$
which is the graph of the rational map $\varphi$. The image of $\id_X\times \varphi$ is contained in $F$. Set
$$
X' := \overline{\mathrm{Image}(\id_X \times \varphi)},
$$
i.e. let $X'$ be the closure of the image of the graph of $\varphi$.

\begin{lemma}\label{quadric lemma}
If $d\ge 2n$ and $(d,n)\ne (2,1)$ then in then the map $\pi$ in (1) is birational. In particular $f$ determines a ``congruence of order one" on $\Quad$.
\end{lemma}

\begin{proof}
The proof of the above lemma is essentially the same as ~\cite[Thm. 4.3]{BastCortDe} and we refer the interested reader to ~\cite[Lemma 4.10]{Thesis}.
\end{proof}

\begin{proof}[Proof of Theorem A]
The following proof follows the proof of ~\cite[Thm. C]{ELU}, and we include it for completeness. Assume for contradiction that there exists a dominant rational map
$$
\phi \cl X \dra \PP^n
$$
with $\delta = \deg(\phi) < d.$ By Theorem 1.5 we know $\delta \ge d-n+1$, and thus by our assumption on $\delta$ we can assume $n\ge 2$. Now consider the divisor
$$
\pi^*X = X' + \sum a_i E_i.
$$
By Lemma 5, the map $\pi$ is birational so we conclude that $\pi_* E_i = 0$. As the $E_i$ are effective divisors, for a fiber $\ell$ of $\psi$ we must have $E_i\cdot [\ell] \ge 0$. We also know
\begin{center}
$X' \cdot [\ell] =\deg(\phi) = \delta$ and $\pi^*X \cdot [\ell] = d.$
\end{center}
The lower bound on $\delta$ implies there exists $E=E_i$ with
$$
0 < c = \deg(\psi|_E) = E\cdot [\ell] \le d-\delta \le n-1.
$$

Now the image $\pi(E)\subset X$ satisfies $e:= \dim(\pi(E))\ge 1$ as every point in $\Quad$ is connected to $\pi(E)$ by a line inside $\Quad$, and the dimension of lines through a single point is $n-1$. Thus the image $\pi(E)$ has covering gonality $\le n-1$, and then by \cite[Prop. 3.7]{Thesis} we have
\begin{equation}
c\ge e+d-2n+1.
\end{equation}
There is another inequality relating $e$ and $c$ which arises from understanding the contribution of $E$ to the effective divisor $K_{F/\Quad}$. As shown in ~\cite[Cor. A.6]{ELU} we have
$$
\ord_E K_{F/\Quad} \ge n-e.
$$
Moreover,
$$
-2= K_F\cdot [\ell] = (K_{F/\Quad}+\pi^*K_\Quad)\cdot[\ell]=(K_{F/\Quad}\cdot [\ell]) -n-1.
$$
Thus
\begin{equation}
n-1 = K_{F/\Quad}\cdot [\ell] \ge \ord_E(K_{F/\Quad})E\cdot [\ell]\ge (n-e)c.
\end{equation}
Now combining equations (2) and (3) as in ~\cite[Pf. of Thm. C]{ELU} or ~\cite[Pf. of Thm. 4.1]{Thesis} contradicts the assumption that $d\ge 2n,$ which completes the proof.
\end{proof}


\section{The degree of irrationality of hypersurfaces in cubic threefolds and fourfolds}

Let $X_d \subset Z \subset \mathbb{P}^{n+2}$ be a complete intersection of type $(3,d)$ in a cubic hypersurface $Z$. 
In this section, we prove Theorem B, which calculates the degree of irrationality of $X_d$ for $n = 2,3$. Our proof depends on Theorem \ref{deg2curve}, which describes the geometry of the fibers of low degree rational maps $X_d \dashrightarrow \mathbb{P}^n$, as well as known theorems about the geometry of Fano varieties of cubic threefolds and cubic fourfolds.

First we give upper and lower bounds on the degree of irrationality of $X$.

\begin{lemma}\label{Thm B lem}
Let $X = X_d\subset Z$ be a smooth divisor in a smooth cubic hypersurface with $X\in |\Oc_Z(d)|$. If $n=2$ or $3$ and $d\ge 5n-2$ then
$$
2(d-n)+2 \leq \irr(X) \leq 2d.
$$
\end{lemma}

\begin{proof}
For the upper bound, we can choose a line contained in $Z$ that meets $X$ in a zero-dimensional subscheme of length $d$. Projection from such a line yields a rational map of degree $2d$.

For the lower bound, consider a dominant rational map $
\phi\colon X \dashrightarrow \PP^n$. For the sake of contradiction, assume $\deg(\phi) \leq 2(d-n) +1.$ Since $\omega_X = \Oc_X(d-n)$, by Proposition \ref{BCD} a general fiber $\xi$ of $\phi$ satisfies $\cb(d-n)$. By Lemma \ref{BCDline}, $\xi$ lies on a line $\ell$. Moreover,
$$
\# \xi = \deg(\phi)\geq d-n+2 > 3,
$$
which implies that $\ell$ must be contained in $Z$. Thus, as in the proof of Lemma \ref{quadric lemma}, we obtain a rational map 
$$
\PP^n \dashrightarrow \Fano(Z).
$$

If $n=2,$ $\Fano(Z)$ is the so-called \textit{Fano surface}, which embeds into its Albanese (see \cite{ClemGriff}). The Albanese is an abelian fivefold and thus contains no rational curves. Thus, any such rational map is constant. This implies that a general point on the surface $X$ lies on a single line, a contradiction. 

If $n=3,$ then the Fano variety is a hyperk\"ahler manifold of dimension 4 (see \cite{BeauvilleDonagi}). The smooth locus of the image of $\PP^n$ in $\Fano(Z)$ must be Lagrangian with respect to the symplectic form. So the image has dimension $\le 2$. If this was possible, then $X$ would be covered by lines. This is a contradiction as $X$ is of general type. Thus, $\deg(\phi) \geq 2(d-n)+2$, as desired.
\end{proof}

\begin{proof}[Proof of Theorem B]

Let $\phi:X \dashrightarrow \PP^n$ be a map of minimum degree $\delta = \irr(X)$. Since $d \geq 5n -2,$
$$
\deg (\phi) \leq 2d \leq \frac{5}{2}(d-n) +1.
$$
By Theorem \ref{deg2curve}, a general fiber $\xi$ is contained in a curve $C$ of degree 2. If $C=\ell_1\cup \ell_2$ is a union of two lines, then by Remark 1.11 each line contains at least $d-n+1\ge 4$ points. Thus both lines are contained in $Z$. Likewise, if $C$ is a smooth plane conic, then $C\cap Z$ contains at least $\#\xi \ge 2d> 6$ points. So again, $C\subset Z$.

First assume $n=2$, and that a general fiber $\xi$ is contained in the union of two lines $C=\ell_1\cup \ell_2\subset Z$. This gives a rational map
$$
\PP^2 \dra \Sym^2(\Fano(Z)).
$$
Since $\Fano(Z)$ embeds into its Albanese, the above rational map yields the following commutative diagram:
\begin{equation}
\begin{tikzcd}
\PP^2 \arrow[d, dashed] \arrow[rrd, bend left=9
, dashed]&& \\
\Sym^2(\Fano(Z)) \arrow[r,hook] &\Sym^2(\Alb(\Fano(Z)))\arrow[swap,r,"\Sigma"]& \Alb(\Fano(Z)).
\end{tikzcd}
\end{equation}
As $\PP^2$ is rationally connected, the image of $\PP^2$ in $\Sym^2(\Alb(\Fano(Z)))$ must be contained in a fiber of $\Sigma$. A fiber of $\Sigma$ is the Kummer variety $K$ of $\Alb(\Fano(Z))$. By \cite[Thm. 1]{Pirola1}, the rational curves on $K$ are rigid, so the closure of the image of $\PP^2$ in $K$ is either a point or a rational curve. Both cases are impossible, the first for dimension reasons. The second case would imply that $X$ is contained in a rational surface, which is impossible as $X$ is a surface of general type.

Therefore, a general fiber $\xi$ of $\phi$ is contained in a smooth plane conic $C\subset Z$. If $\pi = \pi_\xi$ is the plane spanned by $C$, then $\pi\cap Z=C\cup \ell_\xi$, where $\ell_\xi$ is the residual line to $C$ contained in $Z$. Again this determines a rational map
\begin{center}
$\PP^2\dra \Fano(Z)$ by $y \mapsto [\ell_\xi]$ (where $\xi=\phi^{-1}(y)$).
\end{center}
As above, this map must be constant. So all the conics are residual to the same line $\ell_\xi\subset Z$. Thus the map $\phi$ is given by projection from this line up to post composition with a Cremona transformation, and 
$$
\irr(X)= 
\left\{\begin{array}{l}
2d-2 \text{ if $X$ contains a line,}\\
2d \text{ otherwise.} \end{array}\right.
$$

Now assume $n=3$, and for contradiction assume $\delta<2d$. Let $\xi=\phi^{-1}(y)$ be a general fiber which is contained in a degree 2 curve $C\subset Z$. As $\xi$ is general, no component of $C$ is contained in $X$ (because $X$ is of general type). Thus the intersection $C\cap X$ is a 0-dimensional scheme of length $2d$, of which $\delta$ points are accounted for. For a general point $y\in \PP^3$, we can associate to $y$ the \textit{residual} effective 0-cycle $\zeta_y := [C\cap X]-[\xi]$ which has degree $e=2d-\delta$. By Lemma 3.1, $e\le 4$.

We claim that the cycle $\zeta_y$ is not a constant cycle. First, note that the degree 2 curves $C\subset Z$ must sweep out all of $Z$, because they sweep out some uniruled subvariety of $Z$ which contains the general type threefold $X$.

Consider the case when $C=\ell_1\cup \ell_2$ is the union of two lines. A simple argument using transitivity of the monodromy action for the map $\phi$ implies that both lines must have the same number of residual points. As $Z$ is not a cone, for any fixed point $z\in Z$ (and thus for any finite set of points) a general point of $Z$ cannot be connected to $z$ via a line $\ell\subset Z$. Thus $\zeta_y$ cannot be a constant cycle.

In the case $C$ is a smooth plane conic, suppose for contradiction that there is a point $z\in X$ which is contained in $\zeta_y$ for all general $y\in \PP^3$. As in the $n=2$ case, the conic determines a residual line defined by $\ell_y\cup C=\pi\cap Z\subset Z$ where $\pi$ is the plane spanned by $C$. If $z\in \ell_y$ for general $y\in \PP^3$ then the plane spanned by $C$ is contained in the tangent hyperplane to $Z$ at $z$. This means the conics in the family do not sweep out all of $Z$, a contradiction. So we can assume that for a general point $y\in \PP^3$, the residual line $\ell_y$ does not contain $z$. But then the point $z\in X$ and the line $\ell_y$ span the plane $\pi$, and thus the conic $C$ is determined by $\ell_y$ and $z$. This implies that the rational map
$$
\PP^3 \dra \Fano(Z) \text{ which sends } y\mapsto [\ell_y]
$$
is birational onto its image. This is a contradiction as the image must be Lagrangian (see the proof of Lemma 3.1). Therefore, the cycle $\zeta_y$ is not constant.

Let $\lambda\subset \PP^3$ be a line through a general point in $\PP^3$ such that the closure of
$$
\bigcup_{y\in\lambda} \zeta_y\subset X
$$
is positive dimensional. Define the incidence correspondence:
$$
D=\overline{\{(a,y)\in X\times \PP^3 | a\in \zeta_y, y\in \lambda\}}\subset X\times \PP^3.
$$
Then there is a 1-dimensional component $D_0\subset D$ such that neither of the projections $D_0\ra X$ or $D_0 \ra \lambda$ are constant. The projection to $\lambda$ shows that $\gon(D_0) \le \deg(\zeta_y) \le 4$. Thus image of $D_0$ in $X$ is a curve $E$ with gonality $\le 4$. As we assumed $X$ is very general (in particular it is contained in a very general hypersurface in $\PP^5$) by \cite[Propn. 3.8]{ELU},
$$
\gon(E) \ge d-2\cdot4+1=d-7\ge 6,
$$
which is a contradiction.
\end{proof}


\section{Hypersurfaces in the complete intersection of two quadrics in \texorpdfstring{$\PP^5$}{TEXT}}

Let $Z=Q_1 \cap Q_2 \subset \PP^5$ be a smooth intersection of two quadrics and let
$$
X = X_d \subset Z
$$
be a smooth surface in the linear series $X\in |\Oc_Z(d)|$. The goal of this section is to prove Theorem C. That is if $d\ge 8$ then
$$
\begin{array}{c}
\irr(X)= 
\left\{\begin{array}{l}
2d-2 \text{ if $X$ contains a plane conic,}\\
2d-1 \text{ if $X$ contains a line and no conics,}\\
2d \text{ otherwise.}\end{array}\right.
\end{array}
$$
Moreover, we will prove that any map:
$$
\phi\colon X \dra \PP^2
$$
of degree $\le 2d$ is given by projection from a plane in $\PP^4$.

To start, we recall some classical results about the projective geometry of a smooth (2,2) complete intersection Fano threefold. For every such threefold $Z$ there is an associated genus 2 hyperelliptic curve, $C_Z$ which can be defined by an equation given as follows. Let $M_1$ and $M_2$ be the symmetric matrices corresponding to the quadratic forms determined by $Q_1$ and $Q_2$ respectively. Then $C_Z$ is the hyperelliptic curve defined as the compactification of the affine curve:
$$
\left(y^2=\det(M_1 + tM_2)\right) \subset \CC^2.
$$
In particular, the branch points of the hyperelliptic map
$$
h_Z\cl C_Z \xrightarrow{2:1} \PP^1
$$
correspond to singular quadrics $Q_t\in |H^0(I_Z(2))|\cong \PP^1$.

\begin{remark}
The assumption that $Z$ is smooth implies that for all $Q_t\in |I_Z(2)|$, $Q_t$ has at worst isolated singularities. I.e. for all $t$ the matrix
$$
M_t = M_1 + tM_2
$$
has rank $\ge 5$. Moreover, smoothness of $Z$ implies $C_Z$ is smooth.
\end{remark}

Another way to define $C_Z$ is to look at the incidence variety
$$
\inc_Z = \left\{ (P,t) \middle| P\subset Q_t\text{ is a 2-plane in the quadric } Q_t \in|H^0(I_Z(2))|\right\}\subset \mathrm{Gr}(3,6)\times \PP^1.
$$
Then $C_Z$ can be defined as the Stein factorization of the projection to $t\in \PP^1$:
$$
\begin{tikzcd}
&C_Z\arrow[dr,"h_Z"]&\\
\inc_Z \arrow[ur] \arrow[rr] && \PP^1.
\end{tikzcd}
$$

\begin{remark}
The fiber of $\inc_Z \ra \PP^1$ over $t\in \PP^1$ is the Fano variety of planes in $Q_t$, denoted $\Fano(2,Q_t)$. There are two possibilities for $\Fano(2,Q_t)$:
\begin{enumerate}
\item $Q_t$ is smooth, and $\Fano(2,Q_t)\cong \PP^3 \sqcup \PP^3$, or
\item $Q_t$ has an isolated singularity, and $\Fano(2,Q_t) \cong \PP^3$.
\end{enumerate}
\end{remark}

Historically, people have been interested in relating various aspects of the projective geometry of $Z$ to the geometry of the curve $C_Z$. For our purposes the most important result is the following theorem due to Narasimhan and Ramanan.

\begin{theorem}[{\cite[Thm. 5]{NaRaman}}]
Let $Z\subset \PP^5$ be a smooth complete intersection of two quadrics. Let $\Fano(Z)$ be the associated Fano variety of lines in $Z$. Then
$$
\Fano(Z) \cong \Jac(C_Z).
$$
\end{theorem}

Now we proceed to prove Theorem C.

\begin{proposition}
Let $X\in |\Oc_Z(d)|$ be a smooth hypersurface. There exists a rational map:
$$
\phi_0\colon X \dra \PP^2
$$
such that $\deg(\phi_0) =2d$. In particular, $\irr(X) \le 2d$.
\end{proposition}

\begin{proof}
Let $Q_t \in |I_Z(2)|$ be a quadric in the ideal of $Z$ and let $P\subset Q_t$ be a general plane in $Q_t$. Linear projection from $P$ gives a rational map: $\pi_P \colon \PP^5 \dra \PP^2.$ Setting $\phi_0 = \pi_P|_X$, then $\phi_0$ is dominant and $\deg(\phi_0) = 4d-\#(P \cap X) =2d.$
\end{proof}

\begin{remark}
Let $X$ be as above, and $P\subset Q_t$ a plane in a quadric $Q_t\in |I_Z(2)|$. Let $\phi_0 = \pi_P|_X$ be the restriction to $X$ of the linear projection from $P$. There are three possibilites for $\deg(\phi_0)$:
\begin{enumerate}
\item $\deg(\phi_0)=2d$ if the intersection $P\cap X$ is 0-dimensional,
\item $\deg(\phi_0)=2d-1$ if the intersection $P\cap X$ has a single 1-dimensional component which is a line, or
\item $\deg(\phi_0) = 2d-2$ if the intersection $P\cap X$ contains a plane conic.
\end{enumerate}
\end{remark}

Now we would like to prove that if
$$
\phi\colon X \dra \PP^2
$$
is a rational map with $\deg(\phi) \le 2d$, then $\phi$ is given by projection from a plane $P\subset Q_t$ for some $Q_t \in |I_Z(2)|$. We start by applying Theorem 1.9.

\begin{lemma}
Let
$$
\phi\colon X \dra \PP^2
$$
be a dominant rational map with $\delta = \deg(\phi) \le 2d$. Then $\delta \ge 2d-2$. Moreover, if $\xi = \phi^{-1}(p)$ is a general fiber of $\phi$, then $\xi$ is contained in a smooth conic $C\subset Z$.
\end{lemma}

\begin{proof}
By adjunction, the canonical bundle of $X$ is $\omega_X = \Oc_X(d-2)$. Thus $\xi$ satisfies Cayley-Bacharach with respect to the linear series $|\Oc_{\PP^5}(d-2)|$. The assumption that $d\ge 8$ implies that $\delta \le 5/2(d-2)+1$. Thus by Theorem 1.9 we know that one of the following holds:
\begin{enumerate}
\item $\xi$ is contained in a line $\ell \subset \PP^5$,
\item $\xi$ is contained in a union of two lines $\ell_1\cup \ell_2\subset \PP^5$, or
\item $\xi$ is contained in a smooth plane conic $C\subset \PP^5$.
\end{enumerate}
\noindent We now prove that the first 2 cases are impossible. We follow the same argument as in the proof of Theorem B(1).

Assume for contradiction that we are in case (1), i.e. $\xi\subset \ell$. Then as $d\ge 8$ we know that $\delta \ge 6$ by Theorem 1.5. By Bezout's theorem $\ell \subset Z$. Thus, a general point in $\PP^2$ parameterizes a line in $Z$, so we get a rational map:
$$
\PP^2 \dra \Fano(Z)\cong \Jac(C_Z).
$$
This map must be constant as $\Jac(C_Z)$ contains no rational curves. Therefore, we have that every general point in $X$ is contained in a single line, a contradiction. As $F$ is not contained in a line Lemma 1.8 implies that $\delta \ge 2d-2$.

Now assume for contradiction that $\xi$ is in the union of 2 distinct lines $\ell_1$ and $\ell_2$, i.e. assume we are in case (2). Then by Remark 1.11 we have that at least $d-1$ points lie on each line, thus $\ell_1,\ell_2 \subset Z$. Then a general point $p\in \PP^2$ parameterizes a pair of lines $\ell_1\cup \ell_2$, and we get a rational map:
$$
\psi\colon \PP^2 \dra \Sym^2(\Fano(C_Z)) \cong \Sym^2(\Jac(C_Z)).
$$
The image of such a map must lie in a single fiber of the addition map:
$$
\Sigma \colon \Sym^2(\Fano(C_Z)) \ra \Jac(C_Z).
$$
The fibers of $\Sigma$ are a singular Kummer K3 surfaces. In particular, the fibers are not uniruled. Thus the image of $\psi$ has dimension at most 1. This implies $X$ is contained in a ruled surface, which is a contradiction as $X$ is a general type surface.

Finally, assume $\xi$ is contained in a smooth conic $C\subset \PP^5$. As $d\ge 8$, we have $\#(C\cap Z) \ge \delta \ge 8.$ Bezout's theorem implies $C\subset Z$.
\end{proof}

\begin{lemma}
Let $Z$ be a smooth (2,2)-complete intersection in $\PP^5$, let $C\subset Z$ a plane conic, and let $P$ be the plane spanned by $C$. Then there is a unique quadric in the pencil
$$
Q_t \in |I_Z(2)|
$$
such that $P\subset Q_t$.
\end{lemma}

\begin{proof}
First, it is clear that there is at most one such quadric, as a smooth (2,2)-complete intersection in $\PP^5$ contains no planes. Now consider the restriction map:
$$
r\colon H^0(\PP^5,I_Z(2)) \ra H^0(P,I_C(2)).
$$
We have $H^0(\PP^5,I_Z(2))$ is 2-dimensional and $H^0(P,I_C(2))$ is 1-dimensional. The map $r$ is nonzero as $P\not\subset Z$. Thus $r$ is surjective, and the kernel of $r$ is 1-dimensional, spanned by the equation of $Q_t$.
\end{proof}

Now given a rational map $\phi\colon X \dra \PP^2$ with $\deg(\phi)\le 2d$, Lemma 4.6 and Lemma 4.7 imply that a general point $t\in \PP^2$ parameterizes a plane $P_t$ which is contained in a quadric $Q_t\in |I_Z(2)|$. All together, this gives a rational map:
$$
\PP^2 \dra \inc_Z.
$$
Note that as $C_Z$ is a smooth genus 2 curve the composition:
$$
\PP^2 \dra \inc_Z \ra C_Z
$$
must be constant. Thus there is some fixed quadric $Q_t\in|I_Z(2)|$ such that the above rational map factors as
$$
\begin{tikzcd}
\PP^2 \arrow[dr,dashed] \arrow[rr,dashed] && \inc_Z.\\
&\Fano(2,Q_t)\arrow[ur]&
\end{tikzcd}
$$

By Remark 4.2, there are two possibilities for $\Fano(2,Q_t)$. Either
\begin{enumerate}
\item $Q_t$ is smooth, and $\Fano(2,Q_t)\cong \PP^3 \sqcup \PP^3$, or
\item $Q_t$ has an isolated singularity, and $\Fano(Q_t) \cong \PP^3$.
\end{enumerate}
In either case the rational map $\PP^2\dra \Fano(2,Q_t)$ lands in a single $\PP^3$. Let $B$ be the closure of the image of $\PP^2$ in $\PP^3$, and consider the following diagram:
\begin{equation}
\begin{tikzcd}
F \arrow[d] \arrow[r,hook] &G\arrow[r,"\pi"]\arrow[d,"\psi"] &Q_t.\\
B\arrow[r,hook]&\PP^3&
\end{tikzcd}
\end{equation}
Here $G$ is the universal plane over $\PP^3$, and $F$ the family of planes over $B$, i.e. $F=B\times_{\PP^3} G$.

\begin{lemma}
Let
$$
f\colon \PP^2 \dra B\subset \PP^3
$$
be the map induced by $\phi\colon X \dra \PP^2$. Then $f$ is birational.
\end{lemma}

\begin{proof}
Start by resolving the indeterminacy of $f$:
\begin{equation}
\begin{tikzcd}
 B' \arrow[d] \arrow[dr,"f'"] &\\
\PP^2\arrow[r,dashed,"f",swap]  & B.
\end{tikzcd}
\end{equation}
It suffices to show that $f'$ is birational. First we prove that $B$ is a surface, i.e. that $f'$ is generically finite. Note that
$$
\pi^{-1}Z\cap F\ra B
$$
is the family of conics in $Z$ parameterized by $B$, and has dimension $\dim(B)+1$. A general point of $X$ is contained in a conic in this family. Thus, $\pi(\pi^{-1}(Z)\cap F)$ is a subvariety of $Z$ containing the divisor $X$. Moreover $X$ cannot be a component of $\pi(\pi^{-1}(Z)\cap F)$ as $X$ is not uniruled. Therefore, $\pi(\pi^{-1}(Z)\cap F) = Z$ which by a dimension count shows $\dim(B)\ge 2$.

Now assume for contradiction that $\deg(f')\ge 2$. For every general point $x\in B$, let $P_x$ be the plane in $Q_t$ which is parameterized by the point $x$ and let $C_x = P_x \cap Z$ be the smooth conic in $Z$ parameterized by $x$. Note that as $x$ is general, $C_x$ is not contained in $X$ as $X$ is not uniruled. Thus the interesection $C_x\cap X$ is proper. If $\deg(f')\ge 2$ then there are at least two fibers of $\phi$ which are contained in $C_x \cap X$. Then we have
$$
2d=\mathrm{length}(C_x\cap X) \ge \#(C_x\cap X) \ge 2\delta \ge 2(2d-2).
$$
This contradicts the assumption that $d\ge 8$.
\end{proof}

\begin{lemma}
If $B\subset \PP^3$ has degree 1 (i.e. $B$ is a plane) then the congruence $B$ corresponds to the closure of the fibers of a projection from a plane, and thus $\phi$ is birationally equivalent to projection from a plane in $Q_t$.
\end{lemma}

\begin{proof}
It is straightforward to show that the fibers of projection from a plane in $Q_t$ give rise to a plane $B\subset \PP^3$. A parameter count shows that all planes in $\PP^3$ arise this way.
\end{proof}

\begin{proof}[Proof of Theorem C]
By Lemma 4.9, what remains to show is that for any map:
$$
\phi\cl X \dra \PP^2
$$
the corresponding surface $B\subset\PP^3$ is a plane. First, if $x\in Q_t$ is a smooth point, then the fiber
$$
\pi^{-1}(x) \cong \PP^1_x
$$
maps isomorphically onto a line in $\PP^3$. Thus if $x$ is general, then
$$
\deg(B\subset \PP^3)=\#(\psi(\PP^1_x) \cap B) = \#(\PP^1_x \cap \psi^{-1}(B)) = \#(\PP^1\cap F) = \deg(\pi|_F).
$$

Therefore we want to prove that $\delta=\deg(\pi|_F)=1$. Note that the following holds
$$
\pi_*(F \cdot {\pi^{-1}(X)}) = \deg(\pi|_F) \cdot [X],
$$
so our strategy will be to understand the intersection $F\cdot {\pi^{-1}(X)}$ as a cycle.

First we claim that the intersection of these varieties is proper. As $F$ is an irreducible divisor and ${\pi^{-1}(X)}$ is also irreducible, it suffices to show that ${\pi^{-1}(X)} \not\subset F$. This follows because the map:
$$
\psi|_{{\pi^{-1}(X)}}\cl {\pi^{-1}(X)}\ra \PP^3
$$
is surjective, but $\psi(F) = B\subsetneq \PP^3$. Thus $F\cdot {\pi^{-1}(X)}$ is a positive linear combination of subvarieties supported on the intersection $F\cap {\pi^{-1}(X)}$.

Define a rational map
$$
X\dra Q_t \times \PP^3
$$
by sending a general point $x\in X$ to the pair $(x,\phi(x))$. Let $X'$ denote the closure of the image of this map. Note that $G\subset Q_t \times \PP^3$ and moreover $X' \subset (F \cap {\pi^{-1}(X)})\subset G$. In particular, this implies $X'$ is a component of $F\cap {\pi^{-1}(X)}$, and thus we have
$$
F\cdot {\pi^{-1}(X)} = a [X'] +\sum b_i [E_i]
$$
with $a, b_i \ge 0$. Now we can compute
$$
\begin{array}{rcl}
2d[B]&=&\deg(\psi\cl {\pi^{-1}(X)} \ra \PP^3)[\PP^3]\cdot B\\
&=&\psi_*({\pi^{-1}(X)})\cdot B\\
&=&\psi_*({\pi^{-1}(X)}\cdot\psi^*B)\\
&=&\psi_*(a[X'] + \sum b_i[E_i])\\
&=&a\deg(\psi|_{X'})[B]+\sum b_i \deg(\psi|_{E_i}) [B]\\
&=& \left(a\delta+\sum b_i \deg(\psi|_{E_i})\right)[B].
\end{array}
$$
I.e. we have the equality
\begin{equation}
2d=\left(a\delta+\sum b_i \deg(\psi|_{E_i})\right).
\end{equation}
On the right hand side all the terms are positive, except for possibly the $\deg(\psi|_{E_i})$ which can be 0. As $\delta \ge 2d-2$ we know that $a=1$.

Now assume there is an $E_i$ such that
$$
\deg(\pi|_{E_i}\cl E_i \ra X)\ne 0,
$$
i.e. $\pi_*(E_i)\ne 0$. This actually implies that the map
$$
\psi\cl E_i \ra B
$$
is surjective (if this were not the case, the fibers of $\psi|_{E_i}$ would be plane conics, which would imply $X$ is uniruled). But now this gives a correspondence between $X$ and $B$ and then Proposition 1.7 implies that $\deg(\psi|_{E_i})\ge d$. By the assumption $d\ge 8$, this contradicts (7). Therefore, there are no $E_i$ such that $\pi_*(E_i) \ne 0.$

Thus, we have
$$
\delta[X] = \pi_*(F\cdot {\pi^{-1}(X)})= \pi_*([X']) + \sum b_i \pi_*([E_i]) = [X] + 0,
$$
which proves $\delta = 1$.
\end{proof}


\section{Hypersurfaces in Grassmannians}

Let $k\ne 1, m-1$, and let $\Grass = \Grkm\subset \PP$ be the Pl\"ucker embedding of the Grassmannian of $k$-planes in an $m$-dimensional vector space. The aim of this section is to prove Theorem D, that is if
$$
X=X_d \subset \Grass
$$
is a very general hypersurface with $X\in|\Oc_\Grass(d)|$ and $d\ge 3m-5$ then $\irr(X) = d$.

To start we show $\irr(X) \le d$.

\begin{lemma}
If $X\in |\Oc_\Grass(d)|$ then there exists a degree $d$ map
$$
\phi_0 \cl X \dra \PP^n.
$$
\end{lemma}

\begin{proof}
To start we show there is a rational map
$$
p\cl\Grass \dra \PP^n
$$
such that every fiber of $p$ is in a line $\ell\subset \PP$ that is contained in $\Grass$. Choose a one dimensional subspace $\lambda \subset \CC^m$, an $(m-1)$-dimensional subspace $W \subset \CC^m$, and let
$$
T\cl \CC^m \ra (\CC^m/\lambda)
$$
denote the quotient map. Let
$$
\Flkm :=\left\{[U\subset V\subset (\CC^m/\lambda)] \middle| \begin{array}{c}
\text{Where $U$ and $V$ are subspaces of $(\CC^m/\lambda)$}\\
\text{of dimensions $k-1$ and $k$ respectively}
\end{array}\right\}
$$
denote the the partial flag variety of $(\CC^m/\lambda)$. Then we can define a rational map from $\Grass$ to $\Flkm$ as follows:
\begin{center}
$p = p_{\lambda,W}\cl \Grass \dra  \Flkm$.\\
\hspace{1in} $[\Lambda \subset \CC^m] \mapsto [T(\Lambda\cap W)\subset T(\Lambda) \subset (\CC^m/\lambda) ]$
\end{center}

Note that $\Flkm \bir \PP^n$, and two general points $[\Lambda\subset \CC^m]$ and $[\Lambda' \subset \CC^m]$ are in the same fiber of $p$ if and only if $\Lambda'$ satisfies $\Lambda \cap W \subset \Lambda' \subset \Lambda+\lambda.$ It is straightforward to show that the closure of all such $[\Lambda'\subset \CC^m]$ form a line in the Pl\"ucker embedding of $\Grass$. Now set $\phi_0$ equal to the composition
$$
\phi_0 = p|_X\cl X \dra \Flkm \bir \PP^n
$$
By the construction of $\phi_0$, the fiber of $\phi_0$ over a general point in $\PP^n$ is contained in a line $\ell\subset \Grass\subset \PP$. An appropriate choice of $\lambda$ and $W$ will guarantee that $\ell \cap X$ does not meet the base locus of $\phi_0$. Thus we have $\deg(\phi_0) =[\ell]\cdot [X]= d.$
\end{proof}

Now assume for contradiction that there is a dominant rational map
$$
\phi\cl X \dra \PP^n
$$
with $\deg(\phi)\le d-1$. First, we show that all fibers of $\phi$ must lie on lines contained inside $\Grass$.

\begin{lemma}
If $d\ge 3m-5$ then a general fiber of $\phi$ lies on a line $\ell \subset \PP$ which is contained in $\Grass$.
\end{lemma}

\begin{proof}
The proof is the same as the proof of Lemma 2.4. We just remark that $\omega_X = \Oc_X(d-m)$. So by Theorem 1.5, $\deg(\phi) \ge d-m+2$ and every fiber of $\phi$ lies on a line $\ell\subset\PP$. The Grassmannian is cut out by quadrics, so by applying Bezout's theorem we have that $\ell\subset \Grass$.
\end{proof}

Thus a general point in $\PP^n$ parameterizes a line in $\Grass$. This gives rise to a rational map from $\PP^n$ to the Fano variety of lines in $\Grass$, which is $\Flkkm$. As in \S2 we get the following diagram:

\begin{equation}
\begin{tikzcd}
X\arrow[d,hook]& X'\arrow[l]\arrow[d,hook]\\
\Grass & F\arrow[l,"\pi"]\arrow[d,"\psi"]\\
& B\arrow[r,"f"] & \Flkkm.
\end{tikzcd}
\end{equation}

\noindent Here $f\colon B\ra \Flkkm$ is a resolution of the indeterminacy of the map $\PP^n \dra \Flkkm$. The variety $F$ is the corresponding family of lines in $\Grass$ over $B$ (with it's natural projections). Finally, $X'$ is the closure of the image of the rational section $X\dra B$ which sends a point $x$ to $(x,\phi(x))$.

\begin{lemma}
The map $\pi$ is birational, i.e. the map $\phi$ determines a ``congruence of lines of order one" on $\Grass$.
\end{lemma}

\begin{proof}
The proof is identical to the proof of \cite[Theorem 4.3]{BastCortDe}, and the proof of this precise statement can be found in \cite[Lemma 4.17]{Thesis}.
\end{proof}

\begin{proof}[Proof of Theorem D]
Assume for contradiction that there is a map
$$
\phi\colon X \dra \PP^n,
$$
with $\deg(\phi)=\delta\le d-1$. Then we can associate to $\phi$ the fundamental diagram (8). Let
$$
\pi^*(X) = X' + \sum a_i E_i
$$
where the $E_i$ are irreducible exceptional divisors of $\pi$ and $a_i>0$. Let $\ell$ be a fiber of $\psi$. As $X \in |\Oc_\Grass(d)|$ we have
$$
d=[X]\cdot \pi_* [\ell]=\pi^*[X] \cdot [\ell] = \delta + \sum a_i [E_i]\cdot [\ell] = \delta+\sum a_i \deg(\psi|_{E_i}).
$$

By the assumption that $\delta \le d-1$, there must by some $E=E_i$ such that $\deg(\psi|_{E})\ge 1$. Set
$$
c=\deg(\psi|_E).
$$
Using that $\omega_X=\Oc_X(d-m)$, by Theorem 1.5 $\delta\ge d-m+2$, which implies
$$
1\le c\le m-2.
$$

Let $e=\dim(\pi(E))$. As $\pi$ is birational and $c\ge 1$, every point in $\Grass$ lies on a line which intersects the image of $\pi(E)$. A dimension count (\cite[Lemma 4.18]{Thesis}) shows that the union of all lines in $\Grass$ that go through a single point has dimension $m-1$. Thus we have the estimate
$$
e+m-1\ge n+1.
$$
Therefore the image $\pi(E)$ is a subvariety of $X$ which has covering gonality $c\le m-2$ and dimension $e\ge n-m+2$.

However, in \cite[Propn. 3.11]{Thesis} following ideas of \cite{VoisinClemens,EinSubvarsI} and \cite[Propn. 3.8]{ELU} it is proved that if $X\in |\Oc_\Grass(d)|$ is very general then for any subvariety of $X$ with dimension $e$ and covering gonality $c$ we have the inequality:
$$
c\ge e+d-m-n+2.
$$
Plugging in the estimates for $c$ and $e$, and using the assumption that $d\ge 3m-5$ gives
$$
m-2\ge c \ge e+d-m-n+2\ge n-m+2+3m-5-m-n+2 = m-1,
$$
which is a contradiction.
\end{proof}


\section{Hypersurfaces in products of projective space}

Let $\PP=\PP^{m_1}\times \cdots \times \PP^{m_k}$ be a product of $k\ge 2$ projective spaces, and let
$$
X=X_{(d_1,\dots,d_k)}\subset \PP
$$
be a very general hypersurface with $X\in |\Oc_\PP(d_1,\dots,d_k)|$. The goal of this section is to prove Theorem E, i.e. if
$$
\min\{d_i - m_i-1\} \ge \max\{m_i\}
$$
then $\irr(X) =\min\{d_i\}$. Throughout this section we define the following constants:
\begin{itemize}
\item $d:=\min\{d_i\}$,
\item $p:=\min\{d_i-m_i-1\}$,
\item $m:=\max\{m_i\}$, and
\item $n:=\dim(X)=m_1+\cdots+m_k-1$.
\end{itemize}
\noindent Thus the goal is to prove that if $p\ge m$ then $\irr(X) = d.$

To start we show that $\irr(X)\le d$.

\begin{lemma}
There is a degree $d$ rational map
$$
\phi_0\colon X \dra \PP^n.
$$
\end{lemma}

\begin{proof}
It suffices to find a rational map of degree $d$ to any $n$-dimensional rational variety. Without loss of generality assume that $d_1=d.$ Let $x\in \PP^{m_1}$ be a general point in the first projective space. Consider the linear projection from $x$:
$$
\pi_x\colon \PP^{m_1}\dra \PP^{m_1-1}.
$$
Let $\mathrm{pr}_i$ denote the $i$th projection $\mathrm{pr}_i\colon \PP \ra \PP^{m_i}$, and consider the rational map:
$$
\pi_x \times \mathrm{pr}_2\times \cdots \times \mathrm{pr}_k \colon \PP \dra \PP^{m_1-1}\times \PP^{m_2}\times \cdots \times \PP^{m_k}\bir \PP^n.
$$
Let $\phi_0 = \pi_x \times \mathrm{pr}_2\times \cdots \times \mathrm{pr}_k|_X$. If $x$ is chosen generally then $x\times \PP^{m_2}\times \cdots \PP^{m_k} \not\subset X.$ It easily follows that $\deg(\phi_0) = d_1 = d$.
\end{proof}

Now assume for contradiction that there is a dominant rational map
$$
\phi\colon X \dra \PP^n
$$
with $\deg(\phi) =\delta< d$. Let $\PP\subset \PP^N$ be the Segre embedding of $\PP$ defined by $|\Oc_\PP(1,\dots,1)|$.

\begin{lemma}
The fibers of $\phi$ lie on lines in $\PP^N$ which are contained in $\PP.$
\end{lemma}

\begin{proof}
The proof is the same as the proof of Lemma 4. We just remark that
$$
\omega_X \cong \Oc_X(d_1-m_1-1,\dots,d_k-m_k-1)
$$
is $p$-very ample, and that $\PP\subset \PP^N$ is cut out by quadrics.
\end{proof}

\begin{remark}
By the projection formula, any line $\ell\subset \PP\subset \PP^N$ has a unique nonconstant projection
$$
\mathrm{pr}_i \colon \PP \ra \PP^{m_i}.
$$
Thus the Fano variety of lines in $\PP$ is a disjoint union:
$$
\Fano(\PP) = \mathrm{Gr}(2,m_1+1)\times \PP^{m_2}\times \cdots \times \PP^{m_k} \sqcup \cdots \sqcup \PP^{m_1}\times \cdots \times \PP^{m_{k-1}}\times \mathrm{Gr}(2,m_k+1).
$$
\end{remark}

Now by Lemma 6.2, the map $\phi\colon X \dra \PP^n$ induces a rational map
$$
\PP^n \dra \Fano(\PP).
$$
Assume without loss of generality that the image of $\PP^n$ is contained in $\mathrm{Gr}(2,m_1+1)\times \PP^{m_2}\times \cdots \times \PP^{m_k}$. To simplify notation set:
$$
\PP_0 := \PP^{m_2}\times \cdots\times \PP^{m_k}.
$$
Then as in \S 2 or \S 5 we arrive at the following fundamental diagram:
\begin{equation}
\begin{tikzcd}
X\arrow[d,hook]& X'\arrow[l]\arrow[d,hook]\\
\PP & F\arrow[l,"\pi"]\arrow[d,"\psi"]\\
& B\arrow[r,"f"] & \mathrm{Gr}(2,m_1+1) \times \PP_0.
\end{tikzcd}
\end{equation}

As usual $f\colon B \ra \mathrm{Gr}(2,m_1+1)\times \PP_0$ is a resolution of the rational map $\PP^n\dra \mathrm{Gr}(2,m_1+1)\times \PP_0$. The map $\psi\colon F\ra B$ is the family of lines in $\PP$ parameterized by $B$ and the map $\psi\colon F \ra \PP$ is the natural map. Finally, $X'$ is the closure of the image of the rational map
$$
\id \times \phi\colon X \dra F\subset \PP \times B,
$$
in particular $X'\ra X$ is birational.

\begin{lemma}
If $p\ge m$ then the map $\pi\colon F \ra \PP$ is birational, i.e. the map $\phi$ determines a ``congruence of lines of order one" on $\PP$.
\end{lemma}

\begin{proof}
The proof is similar, but slightly more subtle than \cite[Thm. 4.3]{BastCortDe} and we refer the interested reader to \cite[Lem. 4.27]{Thesis}.
\end{proof}

\begin{proof}[Proof of Theorem E]
Assume for contradiction that there is a rational map $\phi\colon X \dra \PP^n$ with
$$
\delta:=\deg(\phi)\le d.
$$
By Lemma 6.2 the fibers of $\phi$ lie on lines inside $\PP$. Using the notation of Remark 6.3, assume without loss of generality that the lines spanned by the fibers of $\phi$ are nonconstant only along the projection to $\PP^{m_1}$.

The case $m_1=1$ is distinct from the other cases. In this case, $\mathrm{Gr}(2,m_1+1)$ is a point and the map
$$
f\colon B \ra \mathrm{Gr}(2,m_1+1) \times \PP_0\cong \PP_0
$$
is birational. It is easy to deduce that the map $\phi$ rationally factors through the projection $X\ra \PP_0$, which has degree $\ge d$, a contradiction.

Now assume that $m_1\ge 2$. Every variety in (9) admits a projection to $\PP_0$, and all of the maps in (9) commute with this projection. This allows us to base change the diagram (9) to consider fibers over a very general point $y\in \PP_0$, which gives the following diagram:

\begin{equation}
\begin{tikzcd}
X_y\arrow[d,hook]& X'_y\arrow[l]\arrow[d,hook]\\
\PP^{m_1} & F_y\arrow[l,"\pi_y"]\arrow[d,"\psi_y"]\\
& B_y\arrow[r,"f"] & \mathrm{Gr}(2,m_1+1).
\end{tikzcd}
\end{equation}

\noindent Because $y$ is general, every variety in (10) is reduced and irreducible, both $\pi_y$ and $\pi_y|_{X_y'}$ are birational maps, and the degree of $\psi|_{X_y'}$ is still $\delta$.

As $X$ was chosen to be very general, $X_y$ is a very general degree $d_1$ hypersurface in $\PP^{m_1}$ with a degree $\delta < d\le d_1$ rational map
$$
\phi_y\colon X_y \dra B_y.
$$
As $F_p$ is rational and is a $\PP^1$ bundle over $B_p$ this implies $B_p$ is rationally connected. The proof of \cite[Thm. C]{ELU} works for dominant rational maps to any rationally connected base. Thus as $\delta < d_1$ and $X_y$ is very general in $\PP^{m_1}$, then \cite[Thm. C]{ELU} implies that $\delta =d_1-1$, $\phi_y$ is projection from a point $x\in X$, and $B_p$ is actually rational.

Returning to diagram (9), the point $x\in X$ allows us to define a section of the generically finite map
$$
\pi^{-1}(X) \ra B,
$$
I.e. there is a component $E$ in $\pi^{-1}(X)$, which is different from $X'$ such that $\psi|_E\colon E\ra B$ has degree 1 (and $\pi(E)$ dominates $\PP_0$). Thus $\pi(E)$ is a uniruled subvariety of $X$ of dimension $n+1-m_1$. Finally, by \cite[Prop. 3.13 \& Prop 3.19]{Thesis} we see that
$$
d_1\ge p+3 > p
$$
which contradicts the degree assumption.
\end{proof}

\section{Open problems}

There are many possibilities for future work. We would like to pose several problems which seem like natural extensions of this paper.

First, let $Z$ be a smooth Fano threefold and let $L$ be an ample line bundle on $Z$. Assume that $L$ is \textit{sufficiently positive} in an appropriate sense.
\begin{problem}
Compute $\irr(X)$ for $X\in |L|$ any smooth surface.
\end{problem}
\noindent The results in this paper, as well as the results from \cite[Thm. 1.3]{BastCortDe}, can be used to compute the degree of irrationality of every sufficiently positive smooth surface in $\PP^3$, $\PP^2\times\PP^1$, $(\PP^1)^3$, any smooth quadric threefold, any smooth cubic threefold, or any smooth (2,2)-complete intersection threefold. In each case the degree of irrationality can be computed in terms of the low degree curves contained in $X$. A natural next step would be to compute the degree of irrationality of smooth surfaces in smooth quartic threefolds $Z\subset \PP^4$, or smooth surfaces in quartic double solids.

It is also natural to ask how the degree of irrationality behaves in families. I.e. assume that
$$
\pi\colon \mathcal{X} \ra T
$$
is a smooth family of complex varieties with relative dimension $n$. How does the function
$$
t\in T \mapsto \irr(\pi^{-1}(t))\in \ZZ
$$
behave? When $n=1$ it is well-known that the gonality of a curve is lower-semicontinuous in families. On the other hand, recently Hasset, Pirutka, and Tschinkel (\cite{HasPirTsch}) constructed a family of varieties such that $\irr(\pi^{-1}(t))$ equals 1 on a dense set (i.e. $\pi^{-1}(t)$ is rational) but is strictly greater then 1 at the very general point $t\in T$. These examples occur when $n\ge 4$. On the other hand, in dimension 2 rationality behaves well in families. We optimistically ask

\begin{question}
Let $\pi\colon\mathcal{X}\ra T$ be a smooth family of complex surfaces. Is the function $\irr(\pi^{-1}(t))$ lower-semicontinuous on $T$?
\end{question}

\noindent This is desirable from a number of perspectives. For example, it would allow us to give some naturally defined and geometrically interesting loci on the moduli of general type surfaces (higher dimensional analogues of Brill-Noether loci). One positive result in a similar direction is the recent work of Kontsevich and Tschinkel \cite{KontTsch} which proves that rationality specializes in families of smooth projective complex varieties.

Finally, we ask if there is a more general Cayley-Bacharach result. Let $\Sc$ be a set of $r$ points in projective space which satisfy the Cayley-Bacharach condition with respect to $|mH|.$

\begin{question}
If $r\le \left( \frac{d+3}{2} \right) m+1,$
is $\Sc$ contained in a degree $d$ curve in projective space?
\end{question}

\noindent Here the ratio $(d+3)/2$ should be thought of as the ratio between the number of general points that degree $d$ plane curves can interpolate (${d+2\choose 2}-1$) and the degree ($d$). Presumably, the proof of such a result would have to be less ad hoc then our proof of Theorem 1.9.

\bibliographystyle{amsalpha} 
\bibliography{hypfano} 

\newcommand{\etalchar}[1]{$^{#1}$}
\providecommand{\bysame}{\leavevmode\hbox to3em{\hrulefill}\thinspace}
\providecommand{\MR}{\relax\ifhmode\unskip\space\fi MR }
\providecommand{\MRhref}[2]{%
  \href{http://www.ams.org/mathscinet-getitem?mr=#1}{#2}
}
\providecommand{\href}[2]{#2}
\begin{thebibliography}{BDPE{\etalchar{+}}15}

\bibitem[BCDP14]{BastCortDe}
F.~Bastianelli, R.~Cortini, and P.~De~Poi, \emph{The gonality theorem of
  {N}oether for hypersurfaces}, J. Algebraic Geom. \textbf{23} (2014), no.~2,
  313--339. \MR{3166393}

\bibitem[BCFS17]{BCFS}
Francesco Bastianelli, Ciro Ciliberto, Flaminio Flamini, and Paola Supino,
  \emph{Gonality of curves on general hypersufaces}, 2017.

\bibitem[BD85]{BeauvilleDonagi}
Arnaud Beauville and Ron Donagi, \emph{La vari\'et\'e des droites d'une
  hypersurface cubique de dimension {$4$}}, C. R. Acad. Sci. Paris S\'er. I
  Math. \textbf{301} (1985), no.~14, 703--706. \MR{818549}

\bibitem[BDPE{\etalchar{+}}15]{ELU}
Francesco Bastianelli, P.~De~Poi, Lawrence Ein, Robert Lazarsfeld, and Brooke
  Ullery, \emph{Measures of irrationality for hypersurfaces of large degree},
  2015.

\bibitem[CG72]{ClemGriff}
C.~Herbert Clemens and Phillip~A. Griffiths, \emph{The intermediate {J}acobian
  of the cubic threefold}, Ann. of Math. (2) \textbf{95} (1972), 281--356.
  \MR{0302652}

\bibitem[Ein88]{EinSubvarsI}
Lawrence Ein, \emph{Subvarieties of generic complete intersections}, Invent.
  Math. \textbf{94} (1988), no.~1, 163--169. \MR{958594}

\bibitem[GK17]{GounKouv}
Frank Gounelas and Alexis Kouvidakis, \emph{Measures of irrationality of the
  fano surface of a cubic threefold}, 2017.

\bibitem[HPT16]{HasPirTsch}
Brendan Hassett, Alena Pirutka, and Yuri Tschinkel, \emph{Stable rationality of
  quadric surface bundles over surfaces}, 2016.

\bibitem[KT17]{KontTsch}
Maxim Kontsevich and Yuri Tschinkel, \emph{Specialization of birational types},
  2017.

\bibitem[NR69]{NaRaman}
M.~S. Narasimhan and S.~Ramanan, \emph{Moduli of vector bundles on a compact
  {R}iemann surface}, Ann. of Math. (2) \textbf{89} (1969), 14--51.
  \MR{0242185}

\bibitem[Pir89]{Pirola1}
Gian~Pietro Pirola, \emph{Curves on generic {K}ummer varieties}, Duke Math. J.
  \textbf{59} (1989), no.~3, 701--708. \MR{1046744}

\bibitem[Sta17]{Thesis}
David Stapleton, \emph{The degree of irrationality of very general
  hypersurfaces in some homogeneous spaces}, Ph.D. thesis, Stony Brook
  University, August 2017.

\bibitem[Voi96]{VoisinClemens}
Claire Voisin, \emph{On a conjecture of {C}lemens on rational curves on
  hypersurfaces}, J. Differential Geom. \textbf{44} (1996), no.~1, 200--213.
  \MR{1420353}

\bibitem[Voi18]{VoisinCovgon}
Claire Voisin, \emph{Chow rings and gonality of general abelian varieties},
  2018.

\end{thebibliography}

\end{document}